\documentclass[12pt]{iopart}

\usepackage{iopams}
\usepackage{hyperref}

\newtheorem{theorem}{Theorem}
\newtheorem{corollary}[theorem]{Corollary}

\newtheorem{remark}[theorem]{Remark}
\newenvironment{proof}[1][Proof]{\noindent\textbf{#1.} }{\ \rule{0.5em}{0.5em}}

\begin{document}

\title{Finite and infinite hypergeometric sums involving the digamma function}

\author{Juan L. Gonz\'{a}lez-Santander}

\address{Department of Mathematics, Universidad de Oviedo, 33007 Oviedo, Spain.}
\ead{gonzalezmarjuan@uniovi.es}
\vspace{10pt}
\begin{indented}
\item[]July 2022
\end{indented}

\begin{abstract}
We calculate some finite and infinite sums containing the digamma function
in closed-form. For this purpose, we differentiate selected reduction
formulas of the hypergeometric function with respect to the parameters
applying some derivative formulas of the Pochhammer symbol. Also, we compare
two different differentiation formulas of the generalized hypergeometric
function with respect to the parameters. For some particular cases, we
recover some results found in the literature. Finally, all the results have
been numerically checked.
\end{abstract}

%
\vspace{2pc}

\section{Introduction}

A large number of finite sums and series involving the digamma function have
been compiled by Hansen \cite{Hansen} and more recently by Brychov \cite{Brychov}.
Some authors have contributed to enhance this compilation, such
as Miller \cite{Miller}, who used reduction formulas of the Kamp\'{e} de F%
\'{e}riet function; and Cvijovi\'{c} \cite{Cvijovic}, who used the
derivative of the Pochhammer symbol. Sums involving the digamma function
occur in the expressions of the derivatives of the Mittag-Leffler function
and the Wright function with respect to parameters \cite{Apelblat1,Apelblat2}%
. Also, they occur in the derivation of asymptotic expansions for
Mellin-Barnes integrals \cite{Paris}.

The aim of this paper is the derivation of several apparently new results by
using also the derivative of the Pochhammer symbol to known reduction
formulas of the hypergeometric function. Nevertheless, for the last result
given in this paper, we use other approach. For this purpose, we compare the
expression of the first derivative of the generalized hypergeometric
function with respect to the parameters given in \cite{Fejzullahu} to the
one given in \cite{Sofostasios}. As a consitency test, for many particular
values of the results obtained, we recover expressions given in the
literature. In adittion, we have checked all the derived expressions with
the aid of MATHEMATICA since sometimes we found some erratums in the
literature.

This paper is organized as follows. In Section \ref{Section: Preliminaires},
we present some basic properties of the Pochhammer symbol, the beta and the
digamma functions. In adittion, we set the notation we use throughout the
paper. In Sections \ref{Section: Finite sums} and \ref{Section: Infinte sums}%
, we derive some results for finite and infinite sums respectively involving
the digamma function. Finally, we collect our conclusions in Section \ref%
{Section: Conclusions}.

\section{Preliminaries}\label{Section: Preliminaires}

The Pochhamer symbol is defined as \cite[Eqn. 18:12:1]{Atlas}%
\begin{equation}
\left( x\right) _{n}=\frac{\Gamma \left( x+n\right) }{\Gamma \left( x\right)
},  \label{Pochhammer_def}
\end{equation}%
where $\Gamma \left( x\right) $ denotes the gamma function. Also, the beta
function, defined as \cite[Eqn. 1.5.3]{Lebedev}%
\begin{eqnarray*}
\mathrm{B}\left( x,y\right) &=&\int_{0}^{1}t^{x-1}\left( 1-t\right) ^{y-1}dt,
\\
&&\mathrm{Re}\,x >0,\ \mathrm{Re}\,y>0,
\end{eqnarray*}%
satisfies the property \cite[Eqn. 1.5.5]{Lebedev}
\begin{equation}
\mathrm{B}\left( x,y\right) =\frac{\Gamma \left( x\right) \Gamma \left(
y\right) }{\Gamma \left( x+y\right) }.  \label{Beta_property}
\end{equation}

For $0\leq z\leq 1$, the incomplete beta function is defined as \cite[Eqn.
58:3:1]{Atlas}%
\begin{equation*}
\mathrm{B}_{z}\left( x,y\right) =\int_{0}^{z}t^{x-1}\left( 1-t\right)
^{y-1}dt.
\end{equation*}

Next, we state some properties of the Pochhammer symbol, i.e.\ the
reflection formula \cite[Eqn. 18:5:1]{Atlas},%
\begin{equation}
\left( -x\right) _{n}=\left( -1\right) ^{n}\left( x-n+1\right) _{n},
\label{Reflection_Pochhammer}
\end{equation}%
the properties \cite[Eqn. 18:5:7\&2:12:3]{Atlas},
\begin{eqnarray}
\left( x\right) _{n+1} &=&x\left( x+1\right) _{n},
\label{property_Pochhammer} \\
\left( \frac{1}{2}\right) _{n} &=&\frac{\left( 2n\right) !}{4^{n}n!},
\label{(1/2)_n}
\end{eqnarray}%
and the differentiation of the Pochhammer symbol \cite[Eqn. 18:10:1]{Atlas}
\begin{equation}
\frac{d}{dx}\left( x\right) _{n}=\left( x\right) _{n}\left[ \psi \left(
x+n\right) -\psi \left( x\right) \right] ,  \label{D[(x)_n]}
\end{equation}%
thus%
\begin{equation}
\frac{d}{dx}\left[ \frac{1}{\left( x\right) _{n}}\right] =\frac{1}{\left(
x\right) _{n}}\left[ \psi \left( x\right) -\psi \left( x+n\right) \right] ,
\label{D[1/(x)_n]}
\end{equation}%
where $\psi \left( x\right) $ denotes the digamma function \cite[Ch. 44]%
{Atlas}%
\begin{equation*}
\psi \left( x\right) =\frac{\Gamma ^{\prime }\left( x\right) }{\Gamma \left(
x\right) },
\end{equation*}%
with the following properties \cite[Eqns. 1.3.3-9]{Lebedev}
\begin{eqnarray}
\psi \left( z+1\right) &=&\frac{1}{z}+\psi \left( z\right) ,
\label{psi(1+z)} \\
\psi \left( 1-z\right) -\psi \left( z\right) &=&\pi \cot \left( \pi z\right)
,  \label{psi(1-z)-psi(z)} \\
\psi \left( z\right) +\psi \left( z+\frac{1}{2}\right) +2\log 2 &=&2\psi
\left( 2z\right) ,  \label{Duplication_psi} \\
\psi \left( 1\right) &=&-\gamma ,  \label{psi(1)} \\
\psi \left( \frac{1}{2}\right) &=&-\gamma -\log 4,  \label{psi(1/2)} \\
\psi \left( n+1\right) &=&-\gamma +H_{n},  \label{psi(1+n)} \\
\psi \left( n+\frac{1}{2}\right) &=&-\gamma -\log 4+2H_{2n}-H_{n},
\label{psi(n+1/2)}
\end{eqnarray}%
where
\begin{equation*}
H_{n}=\sum_{k=1}^{n}\frac{1}{k}
\end{equation*}%
is the $n$-th harmonic number.

Throughout the paper, we adopt the notation \cite[p. 797]{Prudnikov3}%
\begin{equation}
\beta \left( z\right) =\frac{1}{2}\left[ \psi \left( \frac{z+1}{2}\right)
-\psi \left( \frac{z}{2}\right) \right] .  \label{beta_def}
\end{equation}

Also, $_{p}F_{q}\left( z\right) $ denotes the generalized hypergeometric
function, usually defined by means of the hypergeometric series \cite[Sect.
16.2]{NIST}:%
\begin{equation*}
_{p}F_{q}\left( \left.
\begin{array}{c}
\left( a_{p}\right) \\
\left( b_{q}\right)%
\end{array}%
\right\vert z\right) =\,_{p}F_{q}\left( \left.
\begin{array}{c}
a_{1},\ldots ,a_{p} \\
b_{1},\ldots b_{q}%
\end{array}%
\right\vert z\right) =\sum_{k=0}^{\infty }\frac{\left( a_{1}\right)
_{k}\cdots \left( a_{p}\right) _{k}}{\left( b_{1}\right) _{k}\cdots \left(
b_{q}\right) _{k}}\frac{z^{k}}{k!},
\end{equation*}%
whenever this series converge and elsewhere by analytic continuation.
Finally, we use the notation:%
\begin{equation*}
\left( \left( a_{p}\right) \right) _{k}=\left( a_{1}\right) _{k}\cdots
\left( a_{p}\right) _{k}.
\end{equation*}

\section{Finite sums involving digamma function}\label{Section: Finite sums}

\begin{theorem}
The following summation formula holds true:%
\begin{eqnarray}
&&\sum_{k=0}^{n}\left( -1\right) ^{k}{n \choose k}\frac{\left( a\right) _{k}}{%
\left( c\right) _{k}}\psi \left( a+k\right)  \label{Chu_Vandermonde_psi} \\
&=&\frac{\left( c-a\right) _{n}}{\left( c\right) _{n}}\left[ \psi \left(
a\right) -\psi \left( c-a+n\right) +\psi \left( c-a\right) \right] .  \nonumber
\end{eqnarray}
\end{theorem}

\begin{proof}
Chu-Vandermonde summation formula \cite[Corollary 2.2.3]{Andrews} is given by%
\begin{equation}
_{2}F_{1}\left( \left.
\begin{array}{c}
-n,a \\
c%
\end{array}%
\right\vert 1\right) =\frac{\left( c-a\right) _{n}}{\left( c\right) _{n}}%
,\quad n\in
\mathbb{N}
.  \label{Chu_vandermonde}
\end{equation}%
According to (\ref{Reflection_Pochhammer}) and (\ref{Pochhammer_def}), we
have%
\begin{eqnarray*}
_{2}F_{1}\left( \left.
\begin{array}{c}
-n,a \\
c%
\end{array}%
\right\vert 1\right) &=&\sum_{k=0}^{\infty }\frac{\left( -n\right)
_{k}\left( a\right) _{k}}{k!\left( c\right) _{k}} \\
&=&\sum_{k=0}^{\infty }\frac{\left( -1\right) ^{k}\left( n-k+1\right)
_{k}\left( a\right) _{k}}{k!\left( c\right) _{k}} \\
&=&\sum_{k=0}^{\infty }\frac{\left( -1\right) ^{k}\,\Gamma \left( n+1\right)
\left( a\right) _{k}}{k!\,\Gamma \left( n-k+1\right) \left( c\right) _{k}} \\
&=&\sum_{k=0}^{n}{{n \choose k}}\frac{\left( -1\right) ^{k}\left( a\right) _{k}%
}{\left( c\right) _{k}}.
\end{eqnarray*}
Apply (\ref{D[(x)_n]}) to differentiate (\ref{Chu_vandermonde})\ with
respect to the parameter $a$. On the one hand, we have%
\begin{equation}
\frac{\partial }{\partial a}\left[ \frac{\left( c-a\right) _{n}}{\left(
c\right) _{n}}\right] =-\frac{\left( c-a\right) _{n}}{\left( c\right) _{n}}%
\left[ \psi \left( c-a+n\right) -\psi \left( c-a\right) \right] ,
\label{Chu_1}
\end{equation}%
and, on the other hand,%
\begin{eqnarray}
&&\frac{\partial }{\partial a}\left[ _{2}F_{1}\left( \left.
\begin{array}{c}
-n,a \\
c%
\end{array}%
\right\vert 1\right) \right]  \nonumber \\
&=&\sum_{k=0}^{n}{n \choose k}\frac{\left( -1\right) ^{k}}{\left( c\right)
_{k}}\frac{d\left( a\right) _{k}}{da}  \nonumber \\
&=&\sum_{k=0}^{n}\left( -1\right) ^{k} {n \choose k}\frac{\left( a\right) _{k}%
}{\left( c\right) _{k}}\psi \left( a+k\right) -\psi \left( a\right)
\sum_{k=0}^{n} {n \choose k}\frac{\left( -1\right) ^{k}\left( a\right) _{k}}{%
\left( c\right) _{k}}  \nonumber \\
&=&\sum_{k=0}^{n}\left( -1\right) ^{k} {n \choose k}\frac{\left( a\right) _{k}%
}{\left( c\right) _{k}}\psi \left( a+k\right) -\psi \left( a\right) \frac{%
\left( c-a\right) _{n}}{\left( c\right) _{n}}.  \label{Chu_2}
\end{eqnarray}%
Equating (\ref{Chu_1})\ to (\ref{Chu_2}), we obtain (\ref%
{Chu_Vandermonde_psi}), as we wanted to prove.
\end{proof}
\begin{corollary}
For $a=1$, taking into account (\ref{psi(1)}),\ we get%
\begin{eqnarray*}
&&\sum_{k=0}^{n}\frac{\left( -1\right) ^{k}\psi \left( k+1\right) }{\left(
n-k\right) !\left( c\right) _{k}} \\
&=&\frac{\left( c-1\right) _{n}}{n!\left( c\right) _{n}}\left[ -\gamma -\psi
\left( c-1+n\right) +\psi \left( c-1\right) \right] .
\end{eqnarray*}
\end{corollary}

\begin{theorem}
Similarly to (\ref{Chu_Vandermonde_psi}), if we perform the derivative with
respect to the $c$ parameter and apply (\ref{D[1/(x)_n]}), we will obtain%
\begin{eqnarray*}
&&\sum_{k=0}^{n}\left( -1\right) ^{k}{n \choose k}\frac{\left( a\right) _{k}}{%
\left( c\right) _{k}}\psi \left( c+k\right) \\
&=&\frac{\left( c-a\right) _{n}}{\left( c\right) _{n}}\left[ \psi \left(
c+n\right) +\psi \left( c-a\right) -\psi \left( c-a+n\right) \right] .
\end{eqnarray*}
\end{theorem}

\begin{corollary}
For $a=1$, we get%
\begin{eqnarray}
&&\sum_{k=0}^{n}\frac{\left( -1\right) ^{k}\psi \left( c+k\right) }{\left(
n-k\right) !\left( c\right) _{k}}  \label{Sum_corollary} \\
&=&\frac{c-1}{n!\left( c-1+n\right) }\left[ \frac{1}{c-1+n}+\psi \left(
c-1\right) \right] .  \nonumber
\end{eqnarray}
\end{corollary}

\begin{corollary}
Taking the limit $c\rightarrow 1$ in (\ref{Sum_corollary}), and applying (%
\ref{psi(1+z)}), we obtain
\begin{equation*}
\sum_{k=0}^{n}\left( -1\right) ^{k+1} {n \choose k}\psi \left( k+1\right) =%
\frac{1}{n}.
\end{equation*}
\end{corollary}

\begin{theorem}
For $n\in
\mathbb{N}
$, the following finite sum holds true:%
\begin{eqnarray}
&&\sum_{k=0}^{n}\frac{2^{k}\left( a\right) _{k}\left( 2n-k-1\right) !}{%
k!\left( n-k\right) !}\psi \left( a+k\right)  \label{Sum_Qureshi} \\
&=&\frac{2^{2\left( n-1\right) }}{n}\left\{ \left( \frac{1+a}{2}\right) _{n}%
\left[ \psi \left( \frac{1+a}{2}+n\right) +\psi \left( \frac{a}{2}\right)
+\log 4\right] \right.  \nonumber \\
&&+\left. \left( \frac{a}{2}\right) _{n}\left[ \psi \left( \frac{a}{2}%
+n\right) +\psi \left( \frac{1+a}{2}\right) +\log 4\right] \right\} .  \nonumber
\end{eqnarray}
\end{theorem}

\begin{proof}
Apply the reflection formula (\ref{Reflection_Pochhammer})\ and the property
(\ref{(1/2)_n}) to the reduction formula \cite{Qureshi}%
\begin{equation*}
_{2}F_{1}\left( \left.
\begin{array}{c}
-n,a \\
-2n+1%
\end{array}%
\right\vert 2\right) =\frac{1}{\left( \frac{1}{2}\right) _{n}}\left[ \left(
\frac{1+a}{2}\right) _{n}+\left( \frac{a}{2}\right) _{n}\right] ,\quad n\in
\mathbb{N}
,
\end{equation*}%
to arrive at%
\begin{equation}
\sum_{k=0}^{n}\frac{2^{k}\left( a\right) _{k}\left( 2n-k-1\right) !}{%
k!\left( n-k\right) !}=\frac{2^{2n-1}}{n}\left[ \left( \frac{1+a}{2}\right)
_{n}+\left( \frac{a}{2}\right) _{n}\right] .  \label{Qureshi_1}
\end{equation}%
Differentiate (\ref{Qureshi_1})\ with respect to parameter $a$, taking into
account (\ref{Duplication_psi}) to obtain (\ref{Sum_Qureshi}), as we wanted
to prove.
\end{proof}

\begin{corollary}
For $a=1$, taking into account (\ref{psi(1+n)}), (\ref{psi(n+1/2)}) and (\ref%
{(1/2)_n}), Eqn. (\ref{Sum_Qureshi})\ is reduced to%
\begin{eqnarray*}
&&\sum_{k=0}^{n}\frac{2^{k}\left( 2n-k-1\right) !}{\left( n-k\right) !}\psi
\left( k+1\right) \\
&=&\frac{\left( n-1\right) !}{4}\left[ 4^{n}\left( H_{n}-2\gamma \right) +%
{2n \choose n}\left( 2\,H_{2n}-H_{n}-2\gamma \right) \right] ,
\end{eqnarray*}%
or equivalently, reversing the sum order,%
\begin{eqnarray*}
&&\sum_{k=0}^{n}\frac{\left( n+k-1\right) !}{2^{k}k!}\psi \left( n+1-k\right)
\\
&=&\frac{\left( n-1\right) !}{2^{n+2}}\left[ 4^{n}\left( H_{n}-2\gamma
\right) + {2n \choose n}\left( 2\,H_{2n}-H_{n}-2\gamma \right) \right] .
\end{eqnarray*}
\end{corollary}

\section{Infinite sums involving digamma function}\label{Section: Infinte
sums}

\begin{theorem}
For $\mathrm{Re}\left( c-a-b\right) >0$, the following infinite series holds
true:%
\begin{eqnarray}
&&\sum_{k=0}^{\infty }\frac{\left( a\right) _{k}\left( b\right) _{k}}{%
k!\left( c\right) _{k}}\psi \left( a+k\right)  \label{Gauss_Psi} \\
&=&\frac{\Gamma \left( c\right) \Gamma \left( c-a-b\right) }{\Gamma \left(
c-a\right) \Gamma \left( c-b\right) }\left[ \psi \left( c-a\right) -\psi
\left( c-a-b\right) +\psi \left( a\right) \right] .  \nonumber
\end{eqnarray}
\end{theorem}

\begin{proof}
Differentiate Gauss summation formula \cite[Theorem 2.2.2]{Andrews}:
\begin{eqnarray}
&&_{2}F_{1}\left( \left.
\begin{array}{c}
a,b \\
c%
\end{array}%
\right\vert 1\right) =\sum_{k=0}^{\infty }\frac{\left( a\right) _{k}\left(
b\right) _{k}}{k!\left( c\right) _{k}}=\frac{\Gamma \left( c\right) \Gamma
\left( c-a-b\right) }{\Gamma \left( c-a\right) \Gamma \left( c-b\right) },
\label{Gauss_summation} \\
&&\mathrm{Re}\left( c-a-b\right) >0,  \nonumber
\end{eqnarray}%
with respect to the parameter $a$.
\end{proof}

\begin{remark}
In \cite[Addendum. Eqn. 55.4.5.2]{Hansen}, we found an equivalent form, but
with an erratum:\
\begin{eqnarray*}
&&\sum_{k=0}^{\infty }\frac{\left( b\right) _{k}\left( c\right) _{k}}{%
k!\left( a\right) _{k}}\left[ \psi \left( c-k\right) -\psi \left( c\right) %
\right] \\
\neq &&\frac{\Gamma \left( a\right) \Gamma \left( a-b-c\right) }{%
\Gamma \left( a-b\right) \Gamma \left( a-c\right) }\left[ \psi \left(
a-c\right) -\psi \left( a-b-c\right) \right] , \\
&&\mathrm{Re}\left( c+b-a\right) <1,
\end{eqnarray*}%
where we have to change in the sum $\psi \left( c-k\right) $ by $\psi \left(
c+k\right) $. Also, the condition seems to be wrong.
\end{remark}

\begin{corollary}
For the particular case $c=2$ and $b=1/2$, taking into account (\ref%
{psi(1+z)}), (\ref{psi(1-z)-psi(z)}),\ we recover the formula given in \cite[%
Eqn. 6.2.1(67)]{Brychov},
\begin{eqnarray}
&&\sum_{k=0}^{\infty }\frac{\left( a\right) _{k}\left( \frac{1}{2}\right)
_{k}}{k!\left( k+1\right) !}\psi \left( a+k\right)  \label{Brychov_1} \\
&=&\frac{2\,\Gamma \left( \frac{3}{2}-a\right) }{\sqrt{\pi }\Gamma \left(
2-a\right) }\left[ \frac{1}{1-a}+\pi \cot \left( \pi a\right) +2\psi \left(
a\right) -\psi \left( \frac{3}{2}-a\right) \right] ,  \nonumber \\
&&\mathrm{Re\,}a<1.  \nonumber
\end{eqnarray}%
However, from (\ref{Gauss_Psi}), we can extend the validity of (\ref%
{Brychov_1})\ to $\mathrm{Re\,}a<3/2$.
\end{corollary}

\begin{corollary}
For $b>0$, the following expansion of the beta function holds true:
\begin{equation*}
\mathrm{B}\left( a,b\right) =-\sum_{k=0}^{\infty }\frac{\left( -b\right) _{k}%
}{k!}\psi \left( a+k\right) .
\end{equation*}
\end{corollary}

\begin{proof}
Calculate the following limit, taking into account (\ref{psi(1)})\ and (\ref%
{psi(1+z)}):%
\begin{equation}
\lim_{x\rightarrow 0}\frac{\psi \left( x\right) }{\Gamma \left( x\right) }%
=\lim_{x\rightarrow 0}\frac{1}{\Gamma \left( x\right) }\left[ \psi \left(
x+1\right) -\frac{1}{x}\right] =-\lim_{x\rightarrow 0}\frac{1}{\Gamma \left(
x+1\right) }=-1.  \label{Limit_Gauss}
\end{equation}%
Take $c=a$ in (\ref{Gauss_Psi}), and apply (\ref{Limit_Gauss}) and (\ref%
{Beta_property}), to obtain%
\begin{eqnarray*}
&&\sum_{k=0}^{\infty }\frac{\left( -b\right) _{k}}{k!}\psi \left( a+k\right)
\\
&=&\lim_{c\rightarrow a}\frac{\Gamma \left( c\right) \Gamma \left(
c-a+b\right) }{\Gamma \left( c-a\right) \Gamma \left( c+b\right) }\left[
\psi \left( c-a\right) -\psi \left( c-a+b\right) +\psi \left( a\right) %
\right] \\
&=&-\frac{\Gamma \left( a\right) \Gamma \left( b\right) }{\Gamma \left(
a+b\right) }=-\mathrm{B}\left( a,b\right) ,
\end{eqnarray*}
\end{proof}

\begin{remark}
If we differentiate Gauss summation formula (\ref{Gauss_summation}) with
respect to parameter $c$ and we apply (\ref{D[1/(x)_n]}), we will obtain for
$\mathrm{Re}\left( c-a-b\right) >0$,%
\begin{eqnarray*}
&&\sum_{k=0}^{\infty }\frac{\left( a\right) _{k}\left( b\right) _{k}}{%
k!\left( c\right) _{k}}\psi \left( c+k\right) \\
&=&\frac{\Gamma \left( c\right) \Gamma \left( c-a-b\right) }{\Gamma \left(
c-a\right) \Gamma \left( c-b\right) }\left[ \psi \left( c-a\right) +\psi
\left( c-b\right) -\psi \left( c-a-b\right) \right] ,
\end{eqnarray*}%
which is equivalent to \cite[Addendum. Eqn. 55.4.5.1]{Hansen}, but the
condition $\mathrm{Re}\left( a+b-c\right) <1$ seems to be wrong.
\end{remark}

\begin{theorem}
The following series holds true:%
\begin{eqnarray*}
&&\sum_{k=0}^{\infty }\frac{\left( a\right) _{k}\left( 1-a\right) _{k}}{%
2^{k}k!\left( b\right) _{k}}\psi \left( b+k\right) \\
&=&\frac{\sqrt{\pi }\,\Gamma \left( b\right) }{2^{b}\Gamma \left( \frac{a+b}{%
2}\right) \Gamma \left( \frac{b-a+1}{2}\right) }\left[ \psi \left( \frac{a+b%
}{2}\right) +\psi \left( \frac{b-a+1}{2}\right) +\log 4\right] .
\end{eqnarray*}
\end{theorem}

\begin{proof}
Differentiate the summation formula \cite[Eqn. 7.3.7(8)]{Prudnikov3}%
\begin{equation*}
_{2}F_{1}\left( \left.
\begin{array}{c}
a,1-a \\
b%
\end{array}%
\right\vert \frac{1}{2}\right) =\sum_{k=0}^{\infty }\frac{\left( a\right)
_{k}\left( 1-a\right) _{k}}{2^{k}k!\left( b\right) _{k}}=\frac{2^{1-b}\sqrt{%
\pi }\,\Gamma \left( b\right) }{\Gamma \left( \frac{a+b}{2}\right) \Gamma
\left( \frac{b-a+1}{2}\right) },
\end{equation*}%
with respect to parameter $b$.
\end{proof}

\begin{corollary}
Take $a=c$ and apply (\ref{psi(1/2)})\ to obtain%
\begin{equation*}
\sum_{k=0}^{\infty }\frac{\left( 1-a\right) _{k}}{2^{k}k!}\psi \left(
a+k\right) =\frac{\psi \left( a\right) -\gamma }{2^{a}}.
\end{equation*}
\end{corollary}

\begin{theorem}
For $\left\vert z\right\vert <1$, the following series holds true:%
\begin{eqnarray*}
&&\sum_{k=0}^{\infty }\frac{\left( b\right) _{k}\,z^{k}}{\left( k+1\right) !}%
\psi \left( k+b\right) \\
&=&\frac{\left( 1-z\right) ^{1-b}}{z\left( 1-b\right) ^{2}}\left\{ \left(
1-b\right) \log \left( 1-z\right) -\left[ 1-\left( 1-z\right) ^{b-1}\right] %
\left[ 1+\left( 1-b\right) \psi \left( b\right) \right] \right\} .
\end{eqnarray*}
\end{theorem}

\begin{proof}
Differentiate the reduction formula \cite[Eqn. 7.3.1(125)]{Prudnikov3},
\begin{equation*}
_{2}F_{1}\left( \left.
\begin{array}{c}
1,b \\
2%
\end{array}%
\right\vert z\right) =\sum_{k=0}^{\infty }\frac{\left( b\right) _{k}}{\left(
k+1\right) !}z^{k}=\frac{\left( 1-z\right) ^{1-b}-1}{z\left( b-1\right) },
\end{equation*}%
with respect to the parameter $b$.
\end{proof}

\begin{remark}
Taking the limit $b\rightarrow 1$, we arrive at
\begin{eqnarray*}
&&\sum_{k=0}^{\infty }\frac{\,z^{k}}{k+1}\psi \left( k+1\right) =\frac{\log
\left( 1-z\right) }{2z}\left[ 2\gamma +\log \left( 1-z\right) \right] , \\
&&\left\vert z\right\vert <1,
\end{eqnarray*}%
which is equivalent to \cite[Eqn. 6.2.1(2)]{Brychov}.
\end{remark}

\begin{remark}
For $b=2$, and taking into account (\ref{psi(1+n)})$\ $for $n=1$, i.e. $\psi
\left( 2\right) =1-\gamma $, we arrive at%
\begin{eqnarray*}
&&\sum_{k=1}^{\infty }\,z^{k}\psi \left( k+1\right) =\frac{\gamma z+\log
\left( 1-z\right) }{z-1}, \\
&&\left\vert z\right\vert <1,
\end{eqnarray*}%
which is equivalent to \cite[Eqn. 6.2.1(1)]{Brychov}.
\end{remark}

\begin{theorem}
For $\left\vert z\right\vert <1$, the following infinite sum holds true:%
\begin{eqnarray*}
&&\sum_{k=0}^{\infty }\frac{z^{k}\left( a+1\right) _{k}\left( b\right) _{k}}{%
k!\left( a\right) _{k}}\psi \left( k+b\right) \\
&=&\frac{\left[ \psi \left( b\right) -\log \left( 1-z\right) \right] \left[
1-\left( 1-\frac{b}{a}\right) z\right] +z/a}{\left( 1-z\right) ^{1+b}}.
\end{eqnarray*}
\end{theorem}

\begin{proof}
Differentiate the following reduction formula \cite[Eqn. 15.4.19]{NIST}
\begin{equation}
_{2}F_{1}\left( \left.
\begin{array}{c}
a+1,b \\
a%
\end{array}%
\right\vert z\right) =\left[ 1-\left( 1-\frac{b}{a}\right) z\right] \left(
1-z\right) ^{-1-b},  \label{2F1(a+1,b;a;z)}
\end{equation}%
with respect to parameter $b$.
\end{proof}

\begin{corollary}
For the particular case $b=a$, we obtain
\begin{eqnarray*}
&&\sum_{k=0}^{\infty }\frac{z^{k}\left( a+1\right) _{k}}{k!}\psi \left(
k+a\right) \\
&=&\frac{\psi \left( a\right) -\log \left( 1-z\right) +z/a}{\left(
1-z\right) ^{1+a}}.
\end{eqnarray*}
\end{corollary}

\begin{corollary}
For the particular case $b=1$, we obtain%
\begin{eqnarray*}
&&\sum_{k=0}^{\infty }z^{k}\left( a+k\right) \psi \left( k+1\right) \\
&=&\frac{z-\left[ \gamma +\log \left( 1-z\right) \right] \left[ a+\left(
1-a\right) z\right] }{\left( 1-z\right) ^{2}}.
\end{eqnarray*}
\end{corollary}

\begin{theorem}
The following series holds true:\
\begin{eqnarray}
&&\sum_{k=0}^{\infty }\frac{\left( k+1\right) !}{\left( b\right) _{k}}%
2^{-k}\psi \left( k+b\right)  \label{F_1/2b} \\
&=&2\left[ \left( b-1\right) \psi \left( b\right) -1\right] +4\left[
2b-3-\left( b-1\right) \left( b-2\right) \psi \left( b\right) \right] \beta
\left( b-1\right)  \nonumber \\
&&+4\left( b-1\right) \left( b-2\right) \beta ^{\prime }\left( b-1\right) .
\nonumber
\end{eqnarray}
\end{theorem}

\begin{proof}
Differentiate the reduction formula \cite[Eqn. 7.3.7(18)]{Prudnikov3}%
\begin{equation*}
_{2}F_{1}\left( \left.
\begin{array}{c}
1,2 \\
b%
\end{array}%
\right\vert \frac{1}{2}\right) =\sum_{k=0}^{\infty }\frac{\left( k+1\right) !%
}{\left( b\right) _{k}}2^{-k}=2\left( b-1\right) \left[ 1-2\left( b-2\right)
\beta \left( b-1\right) \right] ,
\end{equation*}%
with respect to parameter $b$.
\end{proof}

\begin{remark}
If we differentiate the reduction formula \cite[Eqn. 7.3.7(17)]{Prudnikov3}:%
\begin{equation*}
_{2}F_{1}\left( \left.
\begin{array}{c}
1,1 \\
b%
\end{array}%
\right\vert \frac{1}{2}\right) =\sum_{k=0}^{\infty }\frac{k!}{\left(
b\right) _{k}}2^{-k}=2\left( b-1\right) \beta \left( b-1\right) ,
\end{equation*}%
with respect to parameter $b$, we will obtain \cite[Eqn. 6.2.1(64)]{Brychov}%
\begin{eqnarray}
&&\sum_{k=0}^{\infty }\frac{2^{-k}k!}{\left( b\right) _{k}}\psi \left(
k+b\right)  \label{F_1/2} \\
&=&2\left[ \left( b-1\right) \psi \left( b\right) -1\right] \beta \left(
b-1\right) -2\left( b-1\right) \beta ^{\prime }\left( b-1\right) .  \nonumber
\end{eqnarray}
\end{remark}

\begin{corollary}
Substracting (\ref{F_1/2})\ from (\ref{F_1/2b}), we arrive at%
\begin{eqnarray*}
&&\sum_{k=0}^{\infty }\frac{k\,k!}{\left( b\right) _{k}}2^{-k-1}\psi \left(
k+b\right) \\
&=&\left( b-1\right) \left[ \psi \left( b\right) +\left( 2b-3\right) \beta
^{\prime }\left( b-1\right) \right] \\
&&+\left[ 4b-5-\left( b-1\right) \left( 2b-3\right) \psi \left( b\right) %
\right] \beta \left( b-1\right) -1.
\end{eqnarray*}
\end{corollary}

\begin{theorem}
For $\left\vert z\right\vert <1$, the following infinite sum holds true:%
\begin{eqnarray}
&&\sum_{k=0}^{\infty }\frac{z^{k}\left( a+1\right) _{k}\left( b\right) _{k}}{%
\left( a\right) _{k}\left( c\right) _{k}}\,\psi \left( c+k\right)
\label{Sum_Sofostasios} \\
&=&\psi \left( c-1\right) \,_{3}F_{2}\left( \left.
\begin{array}{c}
1,a+1,b \\
a,c%
\end{array}%
\right\vert z\right)  \nonumber \\
&&+\frac{1}{a\left( 1-z\right) ^{1+b}}\left\{ \frac{\,a+\left( b-a\right) z}{%
c-1}\,\,_{3}F_{2}\left( \left.
\begin{array}{c}
b,c-1,c-1 \\
c,c%
\end{array}%
\right\vert \frac{z}{z-1}\right) \right.  \nonumber \\
&&+\left. \frac{b\left( c-1\right) z}{c^{2}\left( z-1\right) }%
\,\,_{3}F_{2}\left( \left.
\begin{array}{c}
b+1,c,c \\
c+1,c+1%
\end{array}%
\right\vert \frac{z}{z-1}\right) \right\} .  \nonumber
\end{eqnarray}
\end{theorem}

\begin{proof}
On the one hand, consider the reduction formula \cite[Eqn. 7.4.4(94)]%
{Prudnikov3}%
\begin{eqnarray*}
&&_{3}F_{2}\left( \left.
\begin{array}{c}
-k,a,b \\
a+\ell ,b+n%
\end{array}%
\right\vert 1\right) \\
&=&k!\left( a\right) _{\ell }\left( b\right) _{n}\left[ \frac{1}{\left( \ell
-1\right) !\left( a\right) _{k+1}\left( b-a\right) _{n}}\,_{3}F_{2}\left(
\left.
\begin{array}{c}
1-\ell ,a,1+a-b-n \\
1+a+k,1+a-b%
\end{array}%
\right\vert 1\right) \right. \\
&&+\left. \frac{1}{\left( n-1\right) !\left( b\right) _{k+1}\left(
a-b\right) _{\ell }}\,_{3}F_{2}\left( \left.
\begin{array}{c}
1-n,b,1+b-a-\ell \\
1+b+k,1+b-a%
\end{array}%
\right\vert 1\right) \right] ,
\end{eqnarray*}%
for the particular case $\ell =1$, $n=1$, to obtain%
\begin{equation*}
_{3}F_{2}\left( \left.
\begin{array}{c}
-k,a,b \\
a+1,b+1%
\end{array}%
\right\vert 1\right) =\frac{k!a\,b}{\left( a\right) _{k+1}\left( b\right)
_{k+1}}\left[ \frac{\left( a\right) _{k+1}-\left( b\right) _{k+1}}{a-b}%
\right] .
\end{equation*}

Take the limit $a\rightarrow b$ and apply (\ref{D[(x)_n]}) as well as the
property (\ref{property_Pochhammer}), to obtain
\begin{eqnarray}
_{3}F_{2}\left( \left.
\begin{array}{c}
-k,b,b \\
b+1,b+1%
\end{array}%
\right\vert 1\right) &=&\frac{k!\,b^{2}}{\left[ \left( b\right) _{k+1}\right]
^{2}}\lim_{a\rightarrow b}\left[ \frac{\left( a\right) _{k+1}-\left(
b\right) _{k+1}}{a-b}\right]  \nonumber \\
&=&\frac{k!\,b^{2}}{\left[ \left( b\right) _{k+1}\right] ^{2}}\frac{d}{dx}%
\left[ \left( x\right) _{k+1}\right] _{x=b}  \nonumber \\
&=&\frac{k!\,b}{\left( b+1\right) _{k}}\left[ \psi \left( b+1+k\right) -\psi
\left( b\right) \right] .  \label{3F2_resultado_parcial}
\end{eqnarray}

On the other hand, from (\ref{2F1(a+1,b;a;z)}), we have%
\begin{equation}
_{2}F_{1}\left( \left.
\begin{array}{c}
\alpha +k+2,\beta +k+1 \\
\alpha +k+1%
\end{array}%
\right\vert z\right) =\left( 1+\frac{\beta -\alpha }{\alpha +k+1}z\right)
\left( 1-z\right) ^{-2-\beta -k}.  \label{2F1_resultado_parcial}
\end{equation}

Now, equate the results given in \cite{Sofostasios} and \cite{Fejzullahu},%
\begin{eqnarray}
&&D_{b}^{m}\left[ _{p}F_{q+1}\left( \left.
\begin{array}{c}
\left( a_{p}\right) \\
b,\left( b_{q}\right)%
\end{array}%
\right\vert z\right) \right]  \label{Diff_hyper_formulas} \\
&=&\frac{m!\left( -1\right) ^{m}\left( \left( a_{p}\right) \right) _{1}}{%
b^{m+1}\left( \left( b_{q}\right) \right) _{1}} \nonumber \\
&&\sum_{k=0}^{\infty }\frac{%
z^{k+1}\left( \left( a_{p}+1\right) \right) _{k}}{k!\left( k+1\right)
!\left( \left( b_{q}+1\right) \right) _{k}}\,_{m+2}F_{m+1}\left( \left.
\begin{array}{c}
-k,b,\ldots ,b \\
b+1,\ldots ,b+1%
\end{array}%
\right\vert 1\right)  \nonumber \\
&=&m!\left( -1\right) ^{m}z \nonumber \\
&&\sum_{k=0}^{\infty }\frac{\left( -z\right)
^{k}\left( \left( a_{p}\right) \right) _{k+1}}{k!\left( k+1\right) !\left(
\left( b_{q}\right) \right) _{k+1}\left( k+b\right) ^{m+1}}%
\,_{p}F_{q+1}\left( \left.
\begin{array}{c}
\left( a_{p}\right) +k+1 \\
\left( b_{q}\right) +k+1,k+2%
\end{array}%
\right\vert z\right) ,  \nonumber
\end{eqnarray}%
for the particular case $m=1$, $\left( a_{p}\right) =\left( \alpha +1,\beta
,1\right) $ and $\left( b_{q}\right) =\left( \alpha \right) $, to obtain%
\begin{eqnarray*}
&=&\frac{\left( \alpha +1\right) \beta }{b^{2}\alpha }\sum_{k=0}^{\infty }%
\frac{z^{k}\left( \alpha +2\right) _{k}\left( \beta +1\right) _{k}}{k!\left(
\alpha +1\right) _{k}}\,_{3}F_{2}\left( \left.
\begin{array}{c}
-k,b,b \\
b+1,b+1%
\end{array}%
\right\vert 1\right) \\
&=&\sum_{k=0}^{\infty }\frac{\left( -z\right) ^{k}\left( \alpha +1\right)
_{k+1}\left( \beta \right) _{k+1}}{k!\left( \alpha \right) _{k+1}\left(
k+b\right) ^{2}}\,_{2}F_{1}\left( \left.
\begin{array}{c}
\alpha +k+2,\beta +k+1 \\
\alpha +k+1%
\end{array}%
\right\vert z\right) .
\end{eqnarray*}

Next, insert (\ref{3F2_resultado_parcial})\ and (\ref{2F1_resultado_parcial}%
), and simplify the result using (\ref{property_Pochhammer}), to arrive at%
\begin{eqnarray*}
&&\sum_{k=0}^{\infty }\frac{z^{k}\left( \alpha +2\right) _{k}\left( \beta
+1\right) _{k}}{\left( \alpha +1\right) _{k}\left( b+1\right) _{k}}\,\left[
\psi \left( b+1+k\right) -\psi \left( b\right) \right] \\
&=&\frac{1\,}{b\left( \alpha +1\right) \left( 1-z\right) ^{2+\beta }} \nonumber \\%
&&\sum_{k=0}^{\infty }\frac{\left( \beta +1\right) _{k}\left[ \left( b\right)
_{k}\right] ^{2}}{k!\,\left[ \left( b+1\right) _{k}\right] ^{2}}\left[
\,\alpha +k+1+\left( \beta -\alpha \right) z\right] \left( \frac{z}{z-1}%
\right) ^{k}.
\end{eqnarray*}

Grouping terms,%
\begin{eqnarray*}
&&\sum_{k=0}^{\infty }\frac{z^{k}\left( \alpha +2\right) _{k}\left( \beta
+1\right) _{k}}{\left( \alpha +1\right) _{k}\left( b+1\right) _{k}}\,\psi
\left( b+1+k\right) \\
&=&\psi \left( b\right) \sum_{k=0}^{\infty }\frac{z^{k}\left( \alpha
+2\right) _{k}\left( \beta +1\right) _{k}\left( 1\right) _{k}}{k!\left(
\alpha +1\right) _{k}\left( b+1\right) _{k}} \\
&&+\frac{1\,}{b\left( \alpha +1\right) \left( 1-z\right) ^{2+\beta }}\nonumber \\
&&\left\{\,\left( \alpha +1+\left( \beta -\alpha \right) z\right) \sum_{k=0}^{\infty }%
\frac{\left( \beta +1\right) _{k}\left[ \left( b\right) _{k}\right] ^{2}}{%
k!\,\left[ \left( b+1\right) _{k}\right] ^{2}}\,\left( \frac{z}{z-1}\right)
^{k}\right. \\
&& +\left. \frac{\left( \beta +1\right) b^{2}}{\left(
b+1\right) ^{2}}\left( \frac{z}{z-1}\right) \sum_{k=0}^{\infty }\frac{\left(
\beta +2\right) _{k}\left[ \left( b+1\right) _{k}\right] ^{2}}{k!\,\left[
\left( b+2\right) _{k}\right] ^{2}}\,\left( \frac{z}{z-1}\right)
^{k}\right\} ,
\end{eqnarray*}%
and recasting the sums with hypergeometric functions (renaming the
parameters), we finally arrive at (\ref{Sum_Sofostasios}), as we wanted to
prove.
\end{proof}
\begin{remark}
For the particular case $a=b$, and taking into account the reduction formula
\cite[Eqn. 7.3.1(119)]{Prudnikov3}%
\begin{equation*}
_{2}F_{1}\left( \left.
\begin{array}{c}
1,a \\
c%
\end{array}%
\right\vert z\right) =z^{1-c}\left( 1-z\right) ^{c-a-1}\left( c-1\right)
\mathrm{B}_{z}\left( c-1,a-c+1\right) ,
\end{equation*}%
we obtain for $\left\vert z\right\vert <1$%
\begin{eqnarray}
&&\sum_{k=0}^{\infty }\frac{z^{k}\left( a\right) _{k}}{\left( c\right) _{k}}%
\,\psi \left( c+k\right)  \label{Particular_JL} \\
&=&\psi \left( c-1\right) \,z^{1-c}\left( 1-z\right) ^{c-a-1}\left(
c-1\right) \mathrm{B}_{z}\left( c-1,a-c+1\right)  \nonumber \\
&&+\frac{1}{\left( 1-z\right) ^{a}}\left\{ \frac{\,1}{c-1}%
\,\,_{3}F_{2}\left( \left.
\begin{array}{c}
a-1,c-1,c-1 \\
c,c%
\end{array}%
\right\vert \frac{z}{z-1}\right) \right.  \nonumber \\
&&+\left. \frac{\left( c-1\right) z}{c^{2}\left( z-1\right) }%
\,\,_{3}F_{2}\left( \left.
\begin{array}{c}
a,c,c \\
c+1,c+1%
\end{array}%
\right\vert \frac{z}{z-1}\right) \right\} ,  \nonumber
\end{eqnarray}%
which is a non-trivial alternative form of the result given in \cite%
{Cvijovic}:%
\begin{eqnarray}
&&\sum_{k=0}^{\infty }\frac{z^{k}\left( a\right) _{k}}{\left( c\right) _{k}}%
\,\left[ \psi \left( c+k\right) -\psi \left( c\right) \right]
\label{Particular_Cvijovic} \\
&=&\frac{a\,z}{c^{2}\left( 1-z\right) ^{\alpha +1}}\,_{3}F_{2}\left( \left.
\begin{array}{c}
a+1,c,c \\
c+1,c+1%
\end{array}%
\right\vert \frac{z}{z-1}\right) ,  \nonumber \\
a &\in &%
\mathbb{C}
,\left\vert z\right\vert <1.  \nonumber
\end{eqnarray}
\end{remark}

\section{Conclusions}\label{Section: Conclusions}

We have calculated some finite and infinite sums involving the digamma
function differentiating some reduction formulas of the hypergeometric
function with respect to the parameters and applying the differentiation
formulas of the Pochhammer symbol given in (\ref{D[(x)_n]})\ and (\ref%
{D[1/(x)_n]}). It is worth noting that this method can be applied to many
other reduction formulas of hypergeometric and generalized hypergeometric
functions. Here we have only selected some interesting new cases,
some of which have allowed us to detect errors in the literature. Also,
as a consistency test, we have recovered some formulas found in the literature
from some particular cases of the results obtained.

Nevertheless, in (\ref{Sum_Sofostasios}), we have applied other approach,
wherein we have compared the differentiation formulas given in (\ref%
{Diff_hyper_formulas})\ for a particular case of the parameters. This
approach is not as straightforward as the other one. However, note that the
particular case given in (\ref{Particular_JL})\ applying this method
provides a non-trivial alternative form of the result (\ref%
{Particular_Cvijovic})\ found in the literature.

Finally, we point out that all the sums presented in this paper have been
numerically checked with MATHEMATICA\ and they are available at %
\url{https://shorturl.at/CFG24}.

\section*{References}


\providecommand{\newblock}{}

\end{document}